\documentclass[12pt]{article}

\usepackage{amsmath}
\usepackage{amssymb}
\usepackage{amsthm}

\font\smallit=cmti10
\font\smalltt=cmtt10
\font\smallrm=cmr9

\newtheorem{theorem}{Theorem}[section]

\theoremstyle{definition}
\newtheorem{remark}[theorem]{Remark}

\numberwithin{equation}{section}

\makeatletter
    \def\thebibliography#1{\section*{\@mkboth
      {}{}}\list
      {[\arabic{enumi}]}{\settowidth\labelwidth{[#1]}\leftmargin\labelwidth
    \advance\leftmargin\labelsep
    \usecounter{enumi}}
    \def\newblock{\hskip .11em plus .33em minus .07em}
    \sloppy\clubpenalty4000\widowpenalty4000
    \sfcode`\.=1000\relax}
    \makeatother

\begin{document}

\begin{center}
{\bf COMBINATORICS OF RAMANUJAN-SLATER TYPE IDENTITIES}
\vskip 20pt
{\bf James McLaughlin}\\
{\smallit Department of Mathematics, West Chester University, West Chester, PA 19383, USA}\\
{\tt jmclaughl@wcupa.edu}\\
\vskip 10pt
{\bf Andrew V. Sills}\\
{\smallit Department of Mathematical Sciences, Georgia Southern University, Statesboro, GA 30460,
USA}\\
{\tt ASills@GeorgiaSouthern.edu}\\
\end{center}
\vskip 30pt
\centerline{\smallit Received: , Accepted: , Published: } 
\vskip 30pt

\centerline{\bf Abstract}
\noindent
We provide the missing member of a family of four $q$-series
identities related to the modulus 36, the other members having been found by
Ramanujan and Slater.   We  examine combinatorial
implications of the identities in this family, and of some of the
identities we considered in ``Identities of the Ramanujan-Slater type related to the 
moduli 18 and 24," [\emph{J. Math. Anal. Appl.} \textbf{344}/2 (2008) 765--777].

\pagestyle{myheadings}
\markright{\smalltt INTEGERS: \smallrm ELECTRONIC JOURNAL OF COMBINATORIAL NUMBER THEORY \smalltt x (200x), \#Axx\hfill} 

\thispagestyle{empty}
\baselineskip=15pt
\vskip 30pt

\addtocounter{section}{1}
\section*{\normalsize \thesection. Introduction}
The Rogers-Ramanujan identities,
\begin{equation}\label{RR1}
  \sum_{j=0}^\infty \frac{ q^{j^2} }{ (q;q)_j } =
  \underset{k\equiv\pm 1\hskip -3mm\pmod{5}}{\prod_{k\geqq 1}}\frac{ 1}{1-q^k}
\end{equation}
and
\begin{equation}\label{RR2}
  \sum_{j=0}^\infty \frac{ q^{j^2+j} }{ (q;q)_j } =
   \underset{k\equiv\pm 2\hskip -3mm\pmod{5}}{\prod_{k\geqq 1}}\frac{ 1}{1-q^k},
\end{equation}
where
\begin{align*}
  (a;q)_j &:= \prod_{k=0}^{j-1} (1-a q^k)
\end{align*}
were first proved by L.J. Rogers~\cite{R94} in 1894 and later independently rediscovered (without proof) by
S. Ramanujan~\cite[Vol II, p. 33]{M16}.  Many additional ``$q$-series = infinite product" identities were
found by Ramanujan and recorded in his lost notebook~\cite{AB05}, \cite{AB07}.  A large collection
of such identities was produced by L.J. Slater~\cite{S52}.

Just as the Rogers-Ramanujan identities~\eqref{RR1}, \eqref{RR2} are a family of two
similar
identities where the
infinite products are related to the modulus 5, most Rogers-Ramanujan type identities
exist in a family of several similar identities where the sum sides are similar and
the product sides involve some
common modulus.

In most cases Ramanujan and Slater found all members a given family, but in a few cases
they found just one or two members of a family of four or five identities.
In~\cite{MS07}, we found some ``missing" members of families of
identities related to the moduli 18 and 24 where Ramanujan and/or Slater had found
one or two of the family members, as well as two new complete families.

In this paper, we find the missing member in a family of four
identities related to the modulus 36. We  examine combinatorial
implications of the identities in this family, and of some of the
identities we considered in~\cite{MS07}.

\vskip 30pt

\addtocounter{section}{1}
\section*{\normalsize \thesection. Combinatorial Definitions}
Informally, a partition of an integer $n$ is a representation of $n$ as a sum of
positive integers where the order of the summands is considered irrelevant.  Thus the five partitions
of $4$ are $4$ itself, $3+1$, $2+2$, $2+1+1$, and $1+1+1+1$.  The summands
are called the ``parts" of the partition, and since the order of the parts is
irrelevant, $2+1+1$, $1+2+1$, and $1+1+2$ are all considered to be the same partition
of $4$.  It is often convenient to impose a canonical ordering for the parts and to separate
parts with
commas instead of plus signs, and so we make the following definitions:

  A partition $\lambda$ of an integer $n$ into $\ell$ parts is an $\ell$-tuple of positive
integers $( \lambda_1, \lambda_2, \dots, \lambda_\ell)$ where
  \[ \lambda_i \geqq \lambda_{i+1} \mbox{   for $1\leqq i \leqq \ell-1$,  and} \]
  \[ \sum_{i=1}^\ell \lambda_i = n .\]
The number of parts $\ell = \ell(\lambda)$ of $\lambda$ is also called the \emph{length} of $\lambda$.
The sum of the parts of $\lambda$ is called the \emph{weight} of $\lambda$ and is
denoted $|\lambda|$.

 Thus in this notation, the five partitions of $4$ are $(4)$, $(3,1)$, $(2,1,1)$, and $(1,1,1,1)$.

 In~\cite{A07}, G. Andrews considers some of the implications of generalizing the notion of
partition to include the possibility of some negative integers as parts.  We may formalize
this idea with the following definitions:

 A \emph{signed partition} $\sigma$ of an integer $n$ is a partition pair $(\pi, \nu)$ where
\[ n = | \pi | - | \nu|  . \]
We may call $\pi$ (resp. $\nu$) the \emph{positive (resp. negative) subpartition of $\sigma$ }
and $\pi_1, \pi_2, \dots, \pi_{\ell(\pi)}$ (resp. $\nu_1, \nu_2, \dots, \nu_{\ell(\nu)}$) the
\emph{positive (resp. negative) parts} of $\sigma$.

Thus $( (6,3,3,1), (4,2,1,1) )$, which represents $6+3+3+1-1-1-2-4$, is an example of a
signed partition of $5$.  Of course, there are infinitely many unrestricted
signed partitions of any integer, but when we place restrictions on how parts may appear,
signed partitions arise naturally in the study of certain $q$-series.

\begin{remark}
Notice that the way we have defined signed partitions, the ``negative parts" are
positive numbers (which count negatively toward the weight of the signed partition), much as the 
``imaginary part" of a complex number is real.
\end{remark}

\vskip 30pt

\addtocounter{section}{1}
\section*{\normalsize \thesection. Partitions and $q$-series Identities of Ramanujan and Slater}
Using ideas that originated with Euler, MacMahon~\cite[vol. II, ch. III]{M16} and Schur~\cite{S17}
independently realized that~\eqref{RR1} and~\eqref{RR2} imply the following partition identities:

\begin{theorem}[First Rogers-Ramanujan identity---combinatorial version]  For all integers $n$,
the number of partitions $\lambda$ of $n$ where
   \begin{equation} \lambda_i - \lambda_{i+1} \geqq 2 \mbox{  for $1\leqq i \leqq \ell(\lambda)-1$}
   \label{RR1DiffCond}
   \end{equation}
 equals the number of partitions of $n$ into parts congruent to $\pm 1\pmod{5}$.
\end{theorem}
\begin{theorem}[Second Rogers-Ramanujan identity---combinatorial version]  For all integers $n$,
the number of partitions $\lambda$ of $n$ where
   \begin{equation} \lambda_i - \lambda_{i+1} \geqq 2 \mbox{  for $1\leqq i \leqq \ell(\lambda)-1$}
   \label{RR2DiffCond}
   \end{equation}and
   \begin{equation} \label{NoOnes} \lambda_{\ell(\lambda)} > 1, \end{equation}
 equals the number of partitions of $n$ into parts congruent to $\pm 2\pmod{5}$.
\end{theorem}

When studying sets of partitions where the appearance or exclusion of parts is governed by
difference conditions such as~\eqref{RR1DiffCond}, it is often useful to introduce a second
parameter $a$.  The exponent on $a$ indicates the length of a partition being enumerated,
while the exponent on $q$ indicates the weight of the partition.

For example, it is standard to generalize~\eqref{RR2} and~\eqref{RR1} as follows:
\begin{align}
  F_1 (a,q) &:= \sum_{j=0}^\infty \frac{  a^j q^{j^2+j} }{ (q;q)_j } = \frac{1}{(aq;q)_\infty}
    \sum_{k=0}^\infty \frac{ (-1)^j a^{2j} q^{j(5j+3)/2} (a;q)_j (1-aq^{2j+1}) }{  (q;q)_j  } \label{aRR2} \\
  F_2 (a,q) &:= \sum_{j=0}^\infty \frac{  a^j q^{j^2} }{ (q;q)_j } = \frac{1}{(aq;q)_\infty}
    \sum_{k=0}^\infty \frac{ (-1)^j a^{2j} q^{j(5j-1)/2} (a;q)_j (1-aq^{2j}) }{ (1-a) (q;q)_j  }, \label{aRR1}
\end{align} where \[ (a;q)_\infty := \prod_{k=0}^\infty (1-a q^k). \]
It is then easily seen that $F_1(a,q)$ and $F_2(a,q)$ satisfy the following system of
$q$-difference equations:
\begin{align}
   F_1(a,q) &= F_2(aq,q) \label{F1qDE}\\
   F_2(a,q) &= F_1(a,q)  + aq F_1(aq,q). \label{F2qDE}
\end{align}

Notice that there are straightforward combinatorial interpretations to~\eqref{F1qDE} 
and~\eqref{F2qDE}.  Equation~\eqref{F1qDE} states that if we start with the collection
of partitions satisfying~\eqref{RR1DiffCond} and add $1$ to each part (i.e. replace
$a$ by $aq$),  then we obtain the set of partitions that satsify~\eqref{RR2DiffCond}
and~\eqref{NoOnes}; the difference condition is maintained, but the new partitions
will have no ones.   The left hand side of~\eqref{F2qDE} generates partitions that
satisfy~\eqref{RR1DiffCond} while the right hand side segregates these partitions
into two classes: those where no ones appear (generated by $F_1(a,q)$) and those
where a unique one appears (generated by $aq F_1(aq,q)$).

\begin{remark}
It may seem awkward to have the $a$-generalization of the first (resp. second) 
Rogers-Ramanujan
identity labeled $F_2(a,q)$ (resp. $F_1(a,q)$), but this is actually standard practice
(see, e.g. Andrews~\cite[Ch. 7]{A76}).  Here and in certain generalizations, the 
subscript on $F$ corresponds to one more than the maximum number of ones
which can appear in the partitions enumerated by the function.
\end{remark}

\begin{remark} \label{agen}
While the $a$-generalizations are useful for studying the relevant partitions, the price paid for generalizing~\eqref{RR1} to~\eqref{aRR1}  and~\eqref{RR2} to~\eqref{aRR2}
is that the $a$-generalizations no longer have infinite product representations; only in
the $a=1$ cases will Jacobi's triple product identity~\cite[p. 15, Eq. (1.6.1)]{GR04} allow the right hand sides of~\eqref{RR2}
and~\eqref{RR1} to be transformed into infinite products.
\end{remark}

  An exception to Remark~\ref{agen} may be found in one of the identities in Ramanujan's lost
notebook~\cite[Entry 5.3.9]{AB07}; cf.~\cite[Eq. (1.16)]{MS07}:
\begin{equation}
\sum_{j=0}^\infty \frac{ q^{j^2}  (q^3;q^6)_j }{ (q;q^2)_j^2 (q^4;q^4)_j } =
\underset{j\equiv 1\hskip-3mm \pmod{2} \mbox{ \scriptsize{or} } j\equiv \pm 2{
\hskip-3mm\pmod{12}} }{\prod_{j\geqq 1}} \frac{1}{1-q^j} . \label{RamMod12}
\end{equation}
Equation~\eqref{RamMod12} admits an $a$-generalization with an infinite product:
\begin{equation}
\sum_{j=0}^\infty \frac{a^j q^{j^2} (q^3;q^6)_j }
{ (q;q^2)_j (aq;q^2)_j (q^4;q^4)_j } =
\prod_{j\geqq 1} \frac{1+aq^{4j-2} + a^2 q^{8j-4} }{1-aq^{2j-1}} . \label{aRamMod12}
\end{equation}
Notice that the right hand side of~\eqref{aRamMod12} is easily seen to be equal to
\[ \sum_{n, \ell \geqq 0} s(\ell, n ) a^\ell q^n, \]
where $s(\ell, n)$ denotes the number of partitions of $n$
into exactly $\ell$ parts where no even part appears
more than twice nor is divisible by $4$.
Note also that the right hand side of~\eqref{RamMod12}
generates partitions where parts may appear as in
Schur's 1926 partition theorem~\cite{S26} (i.e. partitions into parts congruent
to $\pm 1\pmod{6}$),
dilated by a factor of $2$, along with unrestricted appearances of odd parts.
It is a fairly common phenomenon for a Rogers-Ramanujan type identity to
generate partitions whose parts are restricted according to a well-known
partition theorem, dilated by a factor of $m$, and where nonmultiples of $m$
may appear without restriction.  See, e.g., Connor~\cite{C72} and Sills~\cite{S07}.

A partner to~\eqref{RamMod12} was found by Slater~\cite[p. 164, Eq. (110), corrected]{S52},
cf.~\cite[Eq. (1.19)]{MS07}:
\begin{equation}
\sum_{j=0}^\infty \frac{ q^{j^2+2j}
(q^3;q^6)_j }{ (q;q^2)_j (q;q^2)_{j+1} (q^4;q^4)_j } =
\underset{j\equiv 1\hskip-3mm \pmod{2} \mbox{ \scriptsize{or} } j\equiv \pm 4{
\hskip-3mm\pmod{12}} }{\prod_{j\geqq 1}} \frac{1}{1-q^j} . \label{Slater110}
\end{equation}
An $a$-generalization of~\eqref{Slater110} is
\begin{equation}\label{aSlater110}
\sum_{j=0}^\infty \frac{a^j q^{j^2+2j}  (q^3;q^6)_j }
{ (q;q^2)_j (aq;q^2)_{j+1} (q^4;q^4)_j } =
\prod_{j\geqq 1} \frac{1+aq^{4j} + a^2 q^{8j} }{1-aq^{2j-1}}  =
\sum_{n, \ell \geqq 0} t(\ell,n) a^\ell q^n,
\end{equation}
where $t(\ell,n)$ denotes the number of partitions of $n$ into $\ell$ parts
where even parts appear at most twice and are divisible by $4$.

\begin{remark}
An explanation as to why~\eqref{RamMod12} and~\eqref{Slater110} admit
$a$-generalizations which include infinite products and~\eqref{RR1} and~\eqref{RR2}
do not, may be found in the theory of basic hypergeometric series.
The Rogers-Ramanujan identities~\eqref{RR1} and~\eqref{RR2} arise as
limiting
cases of Watson's $q$-analog of Whipple's theorem~\cite{W29},\cite[p. 43, Eq. (2.5.1)]{GR04};
see~\cite[pp. 44--45, \S2.7]{GR04}.
In contrast,~\eqref{aRamMod12}
and~\eqref{aSlater110} are special cases of Andrews's $q$-analog of
Bailey's ${}_2 F_1 (\frac 12)$ sum~\cite[p. 526, Eq. (1.9)]{A73},\cite[p. 354, (Eq. II.10)]{GR04}.
\end{remark}

\begin{remark}
S. Corteel and J. Lovejoy interpreted~\eqref{RamMod12} and~\eqref{Slater110}
combinatorially using overpartitions in~\cite{CL07}.
\end{remark}

\vskip 30pt

\addtocounter{section}{1}
\section*{\normalsize \thesection. A Family of Ramanujan and Slater}
\addtocounter{subsection}{1}
\subsection*{\normalsize \thesubsection. A long-lost relative}
Let us define
\[ Q(w,x) := (-wx^{-1}, -x, w; w)_\infty (w x^{-2}, wx^2; w^2)_\infty, \]
where
\[  (a_1, a_2, \dots, a_r;w)_\infty := \prod_{k=1}^r (a_k; w)_\infty. \]
Then it is clear that an identity is missing from the family
\begin{align}
\sum_{j=0}^\infty \frac{ q^{2j(j+2) } (q^3;q^6)_j }{ (q^2;q^2)_{2j+1} (q;q^2)_j }
&= \frac{ Q(q^{18},q^7)}{(q^2;q^2)_\infty} \mbox{ \quad (Slater~\cite[Eq. (125)]{S52})} \label{Slater125}\\
\sum_{j=0}^\infty \frac{ q^{2j(j+1) } (q^3;q^6)_j }{ (q^2;q^2)_{2j+1} (q;q^2)_j }
&= \frac{ Q(q^{18},q^5)}{(q^2;q^2)_\infty} \mbox{ \quad (Slater~\cite[Eq. (124)]{S52})}\label{Slater124}\\
\sum_{j=0}^\infty \frac{ q^{2j^2 } (q^3;q^6)_j }{ (q^2;q^2)_{2j} (q;q^2)_j }
&= \frac{ Q(q^{18},q^3)}{(q^2;q^2)_\infty} \mbox{ \quad (Ramanujan~\cite[Entry 5.3.4]{AB07})}.
\label{RamMod36}
\end{align}

The following identity completes the above family:
\begin{equation}
\sum_{j=0}^\infty \frac{ q^{2j(j+1) } (q^3;q^6)_j }{ (q^2;q^2)_{2j} (q;q^2)_{j+1} }
= \frac{ Q(q^{18},q)}{(q^2;q^2)_\infty}. \label{NewMod36}
\end{equation}
\begin{theorem}
Identity~\eqref{NewMod36} is valid.
\end{theorem}
\begin{proof}
We show that \eqref{Slater124}$+q\times$\eqref{Slater125} $=$
\eqref{NewMod36}. For the series side,
\begin{align*}
\sum_{j=0}^\infty \frac{ q^{2j(j+1) } (q^3;q^6)_j }{
(q^2;q^2)_{2j+1} (q;q^2)_j }+q\sum_{j=0}^\infty \frac{ q^{2j(j+2) }
(q^3;q^6)_j }{ (q^2;q^2)_{2j+1} (q;q^2)_j }&=\sum_{j=0}^\infty
\frac{ q^{2j(j+1) } (q^3;q^6)_j (1+q^{2j+1})}{ (q^2;q^2)_{2j+1}
(q;q^2)_j }\\ &=\sum_{j=0}^\infty \frac{ q^{2j(j+1) } (q^3;q^6)_j }{
(q^2;q^2)_{2j} (q;q^2)_{j+1} }.
\end{align*}
For the product side, we make use of the quintuple product identity:
\[
Q(w,x)=(w x^3,w^2x^{-3},w^3;w^3)_{\infty}+x (w x^{-3},w^2
x^3,w^3;w^3)_{\infty}.
\]
Hence
\begin{align*}
Q(q^{18},q^5)+q Q(q^{18},q^7) &=
(q^{33},q^{21},q^{54};q^{54})_{\infty}
+q^{5}(q^{3},q^{51},q^{54};q^{54})_{\infty}\\
&\hspace{60pt} +q((q^{39},q^{15},q^{54};q^{54})_{\infty}
+q^{7}(q^{-3},q^{57},q^{54};q^{54})_{\infty})\\
&= (q^{33},q^{21},q^{54};q^{54})_{\infty}
+q^{5}(q^{3},q^{51},q^{54};q^{54})_{\infty}\\
&\hspace{60pt} +q((q^{39},q^{15},q^{54};q^{54})_{\infty}
-q^{4}(q^{51},q^{3},q^{54};q^{54})_{\infty})\\
&=(q^{21},q^{33},q^{54};q^{54})_{\infty}+q(q^{15},q^{39},q^{54};q^{54})_{\infty}
=Q(q^{18},q).
\end{align*}
The result now follows.
\end{proof}

\addtocounter{subsection}{1}
\subsection*{\normalsize \thesubsection. Combinatorial Interpretations}
We interpret~\eqref{RamMod36} combinatorially.
\begin{theorem}\label{RamMod36Comb}
The number of signed partitions $\sigma = (\pi, \nu)$ of $n$, where
  \begin{itemize}
     \item $\ell(\pi)$ is even, and each positive part is even and $\geqq \ell(\pi)$, and
     \item the negative parts are odd, less than $\ell(\pi)$, and may appear at most twice
  \end{itemize}
equals the number of (ordinary) partitions of $n$ into parts congruent to $\pm 2, \pm 3, \pm 4,
\pm 8 \pmod{18}$.
\end{theorem}

\begin{proof}
Starting with the left hand side of~\eqref{RamMod36}, we find
\begin{align*}
\sum_{j=0}^\infty \frac{ q^{2j^2 } (q^3;q^6)_j }{ (q^2;q^2)_{2j} (q;q^2)_j }
&= \sum_{j=0}^\infty \frac{ q^{2j^2} \prod_{k=1}^j (1+ q^{2k-1} + q^{4k-2}) }{(q^2;q^2)_{2j} } \\
&=\sum_{j=0}^\infty \frac{ q^{2j^2} \prod_{k=1}^j q^{4k-2}
(1+ q^{-(2k-1)} + q^{-(4k-2)}) }{(q^2;q^2)_{2j} } \\
&=\sum_{j=0}^\infty \frac{ q^{2j^2 + 4(1+2+\cdots+j) -2j} \prod_{k=1}^j
(1+ q^{-(2k-1)} + q^{-(4k-2)}) }{(q^2;q^2)_{2j} } \\
&=\sum_{j=0}^\infty \frac{ q^{4j^2}  }{(q^2;q^2)_{2j} } \times \prod_{k=1}^j (1+ q^{-(2k-1)} + q^{-(4k-2)})\
\end{align*}
Notice that \[ \frac{1}{(q;q)_{2j} } \] is the generating function for partitions into at most $2j$
parts, thus
\[ q^{2j^2} \frac{1}{(q;q)_{2j} }  =   
\frac{q^{\overbrace{j+j+j+\cdots+j}^{\mbox{\scriptsize{ $2j$ terms}}}} }{(q;q)_{2j} } \]
is the generating function for partitions into exactly $2j$ parts, where each part is at least $j$.
Thus
\[ q^{4j^2} \frac{1}{(q^2;q^2)_{2j} }  \]
is the generating function for partitions into exactly $2j$ parts, each of which is even and
at least $2j$.
  Also, $\prod_{k=1}^j (1+ q^{-(2k-1)} + q^{-(4k-2)})$ is the generating function for signed
partitions into odd negative parts $<2j$ and appearing at most twice each.
Summing over all $j$, we find that the left hand side of~\eqref{RamMod36} is the generating
function for signed partitions $\sigma = (\pi, \nu)$ of $n$, where
$\ell(\pi)$ is even, and each positive part is even and $\geqq \ell(\pi)$, and the negative parts are odd, less than $\ell(\pi)$, and may appear at most twice.

 Now the RHS of~\eqref{RamMod36} is
\begin{equation*}
\frac{ Q(q^{18},q^3)}{(q^2;q^2)_\infty} = \frac{ (-q^3, -q^{15};
q^{18}; q^{18})_\infty (q^{12}, q^{24};
q^{36})_\infty}{(q^2;q^2)_\infty} = \underset{i\equiv \pm 2, \pm 3,
\pm 4, \pm 8 \hskip -3mm \pmod{18}}{\prod_{i\geqq
1}}\frac{1}{1-q^i},
\end{equation*}
which is clearly the generating function for partitions into parts congruent to
$\pm 2, \pm 3, \pm 4, \pm 8\pmod{18}$.
\end{proof}

 \begin{remark} Andrews provided a different combinatorial interpretation of~\eqref{RamMod36}
 in~\cite[p. 175, Theorem 2]{A81}.
 \end{remark}

Following the ideas in the proof of Theorem~\ref{RamMod36Comb}, the analogous
combinatorial interpretation
of Identity~\eqref{Slater124} is as follows.
\begin{theorem}
The number of signed partitions $\sigma = (\pi, \nu)$ of $n$, where
  \begin{itemize}
     \item $\ell(\pi)$ is odd, and each positive part is even and $\geqq \ell(\pi)-1$, and
     \item the negative parts are odd, less than $\ell(\pi)$, and may appear at most twice
  \end{itemize}
equals the number of (ordinary) partitions of $n$ into parts
congruent to $\pm 2, \pm 4, \pm 5, \pm 6 \pmod{18}$.
\end{theorem}

We next interpret \eqref{Slater125} combinatorially. Note that the
theorem equates the number in a certain class of signed partitions
of $n+1$ with the number in a certain class of regular partitions of
$n$.
\begin{theorem}
The number of signed partitions $\sigma = (\pi, \nu)$ of $n+1$,
where
  \begin{itemize}
     \item $\ell(\pi)$ is odd, and each positive part is odd and $\geqq \ell(\pi)$, and
     \item the negative parts are odd, less than $\ell(\pi)$, and may appear at most twice
  \end{itemize}
equals the number of (ordinary) partitions of $n$ into parts
congruent to $\pm 2, \pm 6, \pm 7, \pm 8 \pmod{18}$.
\end{theorem}

\begin{proof}
The proof is similar to that of Theorem \ref{RamMod36Comb}, except
that
\[
\sum_{j=0}^\infty \frac{ q^{2j(j+2) } (q^3;q^6)_j }{
(q^2;q^2)_{2j+1} (q;q^2)_j }= \frac{1}{q}\sum_{j=0}^\infty \frac{
q^{4j^2+4j+1}  }{(q^2;q^2)_{2j+1} } \times \prod_{k=1}^j (1+
q^{-(2k-1)} + q^{-(4k-2)}),
\]
and since
\[
4j^2+4j+1=\underbrace{(2j+1)+(2j+1)+\dots + (2j+1)}_{2j+1 \text{
terms }},
\]
it follows that
\[
\frac{ q^{4j^2+4j+1}  }{(q^2;q^2)_{2j+1} }
\]
is the generating function for partitions into exactly $2j+1$ parts,
each of which is odd and at least $2j+1$.
\end{proof}

Lastly, we give a combinatorial interpretation of \eqref{NewMod36}.
\begin{theorem}
The number of signed partitions $\sigma = (\pi, \nu)$ of $n+1$,
where
  \begin{itemize}
   \item   $\pi$ contains an odd positive part $m$ (which may be
repeated), exactly $m-1$ positive even parts, all $\geqq m-1$, and
 \item negative parts are all odd, $<m$, and appear at most twice,
  \end{itemize}
equals the number of (ordinary) partitions of $n$ into parts
congruent to $\pm 1, \pm 4, \pm 6, \pm 8 \pmod{18}$.
\end{theorem}

\begin{proof}
This time
\[
\sum_{j=0}^\infty \frac{ q^{2j(j+1) } (q^3;q^6)_j }{ (q^2;q^2)_{2j}
(q;q^2)_{j+1} }= \frac{1}{q}\sum_{j=0}^\infty \frac{ q^{4j^2}
}{(q^2;q^2)_{2j} } \frac{q^{2j+1}}{1-q^{2j+1}}\times \prod_{k=1}^j
(1+ q^{-(2k-1)} + q^{-(4k-2)}),
\]
so that, as before,
\[
\frac{ q^{4j^2}  }{(q^2;q^2)_{2j} }
\]
is the generating function for partitions into exactly $2j$ parts,
each of which is even and at least $2j$, and
\[
\frac{q^{2j+1}}{1-q^{2j+1}}
\]
generates partitions consisting of the part $2j+1$ and containing at
least one such part.
\end{proof}

\addtocounter{section}{1}
\section*{\normalsize \thesection. Combinatorial Interpretations of a Family of
Mod 18 Identities}
In~\cite{MS07}, we presented the following family of
Rogers-Ramanujan-Slater type identities related to the modulus 18:
{\allowdisplaybreaks
\begin{gather}
\sum_{j=0}^\infty \frac{  q^{j(j+1)} (-1;q^3)_j}{ (-1;q)_j
(q;q)_{2j} } = \frac{(q,q^8,q^9;q^9)_\infty
(q^7,q^{11};q^{18})_\infty}
{(q;q)_\infty} \label{m18-1}\\
\sum_{j=0}^\infty \frac{  q^{j^2} (-1;q^3)_j}{ (-1;q)_j (q;q)_{2j} }
= \frac{(q^2,q^7,q^9;q^9)_\infty (q^5,q^{13} ; q^{18})_\infty}
{(q;q)_\infty} \label{m18-2} \\
\sum_{j=0}^\infty \frac{  q^{j(j+1)} (-q^3;q^3)_j}{ (-q;q)_j
(q;q)_{2j+1} } = \frac{(q^3,q^6,q^9;q^9)_\infty
(q^3,q^{15};q^{18})_\infty}{(q;q)_\infty}
\label{m18-3} \\
\sum_{j=0}^\infty \frac{  q^{j(j+2)} (-q^3;q^3)_j } { (q^2;q^2)_j
(q^{j+2};q)_{j+1} } = \frac{(q^4,q^5,q^9;q^9)_\infty
(q,q^{17};q^{18})_\infty}{(q;q)_\infty} \label{m18-4}.
\end{gather}
}

We give a combinatorial interpretation of \eqref{m18-4}.

\begin{theorem}\label{t-4}
The number of signed partitions $\sigma = (\pi, \nu)$ of $n+2$,
wherein
  \begin{itemize}
     \item $\pi_1$, the largest positive part, is even,
     \item the integers $1,2,\dots, \frac{\pi_1}{2}-1$ all appear
     an even number of times and at least twice,
     \item the integer $\frac{\pi_1}{2}$ does not appear,
     \item the integers $\frac{\pi_1}{2}+1, \frac{\pi_1}{2}+2, \dots, \pi_1$
     all appear at least once, and
     \item there are exactly $\frac{\pi_1}{2}-1$
negative parts, each  $\equiv 1 \pmod{3}$ and $\leqq \frac{3\pi_1}{2}-2$,
with the parts greater than 1 occurring at most once
  \end{itemize}
equals the number of (ordinary) partitions of $n$ into parts
congruent to $\pm 2, \pm3, \pm 6, \pm 7, \pm 8 \pmod{18}$.
\end{theorem}

\begin{proof}
We consider the general term on the left side of \eqref{m18-4}.
\begin{align*}
\frac{  q^{j(j+2)} (-q^3;q^3)_j } { (q^2;q^2)_j (q^{j+2};q)_{j+1} }
&=
\frac{q^{j^2+2j}}{(q^2;q^2)_j}\frac{q^{(3j^2+3j)/2}}{\prod_{k=0}^{j}(1-q^{j+2+k})}
\prod_{k=1}^{j}(1+q^{-3k})\\
&=
\frac{1}{q^2}\frac{q^{j^2+j}}{(q^2;q^2)_j}\frac{q^{(3j^2+7j+4)/2}}{\prod_{k=0}^{j}(1-q^{j+2+k})}
q^{-j}\prod_{k=1}^{j}(1+q^{-3k})
\end{align*}
The factors
\[
\frac{q^{j^2+j}}{(q^2;q^2)_j}= \frac{q^{2+4+6+\dots +
2j}}{(q^2;q^2)_j}
\]
generates parts in $\{2,4,6,\dots,2j\}$ where
each part appears at least once.  Then
by mapping each even part
$2r$ to $r+r$,  we have parts $\{ 1,2,3, \dots, j \}$
where each part appears an even number of times
and at least twice.

The factors
\[
\frac{q^{(3j^2+7j+4)/2}}{\prod_{k=0}^{j}(1-q^{j+2+k})}
=\frac{q^{(j+2)+(j+3)+\dots + (2j+2)}}{\prod_{k=0}^{j}(1-q^{j+2+k})}
\]
generates partitions from the parts $\{j+2, j+3, \dots, 2j+1, 2j+2\}$
and where each part appears at least once. Lastly,
\[
q^{-j}\prod_{k=1}^{j}(1+q^{-3k})
\]
is the generating function for signed partitions with negative
parts that are congruent to 1 modulo 3, $\leq 3j+1$, the parts greater
than 1 occur at most once, and the total number of parts is $j$ (the
number of 1's being $j$ minus the number of other parts).

Upon summing over $j\geq 0$, we get that
\[
\sum_{j=0}^{\infty}\frac{q^{j^2+j}}{(q^2;q^2)_j}\frac{q^{(3j^2+7j+4)/2}}{\prod_{k=0}^{j}(1-q^{j+2+k})}
q^{-j}\prod_{k=1}^{j}(1+q^{-3k})
\]
is the generating function for signed partitions with the properties
itemized in the statement of the theorem.

The right side of \eqref{m18-4} is
\[
\frac{(q^4,q^5,q^9;q^9)_\infty
(q,q^{17};q^{18})_\infty}{(q;q)_\infty}= \underset{i\equiv \pm 2,
\pm 3, \pm 6, \pm 7, \pm 8 \hskip -3mm \pmod{18}}{\prod_{i\geqq
1}}\frac{1}{1-q^i},
\]
which is the generating function for partitions into parts congruent
to $\pm 2,$ $\pm 3, \pm 6, \pm 7, \pm 8\,\, \hskip -3mm \pmod{18}$.
\end{proof}

The corresponding combinatorial interpretation of \eqref{m18-2} is
given by the following theorem.

\begin{theorem}\label{t-2}
The number of signed partitions $\sigma = (\pi, \nu)$ of $n$, where
  \begin{itemize}
     \item $\pi_1$, the largest  positive part, is even,
     \item the integers $1,2,\dots, \frac{\pi_1}{2}-1$ all appear
     an even number of times and at least twice,
     \item the integers $\frac{\pi_1}{2}, \frac{\pi_1}{2}+1, \dots, \pi_1$
     all appear at least once, and
     \item there are exactly $\frac{\pi_1}{2}-1$
negative parts, each  $\equiv 2 \pmod{3}$ and $\leqq \frac{3\pi_1}{2}-1$,
with the parts greater than 2 occurring at most once,
  \end{itemize}
equals the number of (ordinary) partitions of $n$ into parts
congruent to $\pm 1, \pm3, \pm 4, \pm 6, \pm 8 \pmod{18}$.
\end{theorem}

\begin{proof}
The proof is similar to that of Theorem \ref{t-4}, except we rewrite
the general term on the left side of \eqref{m18-2} as follows
\begin{align*}
\frac{  q^{j^2} (-1;q^3)_j } { (-1;q)_j (q;q)_{2j} } &=
\frac{q^{j^2}}{(q^2;q^2)_{j-1}}\frac{q^{(3j^2-3j)/2}}{\prod_{k=0}^{j}(1-q^{j+k})}
\prod_{k=1}^{j-1}(1+q^{-3k})\\
&=
\frac{q^{j^2-j}}{(q^2;q^2)_{j-1}}\frac{q^{(3j^2+3j)/2}}{\prod_{k=0}^{j}(1-q^{j+k})}
q^{-2j}\prod_{k=1}^{j-1}(1+q^{-3k}).
\end{align*}
\end{proof}

The identity at \eqref{m18-1} may be interpreted combinatorially as
follows.

\begin{theorem}\label{t-1}
The number of signed partitions $\sigma = (\pi, \nu)$ of $n$, where
  \begin{itemize}
     \item $\pi_1$, the largest  positive part, is even,
     \item the integers $1,2,\dots, \frac{\pi_1}{2}-1$ all appear
     an even number of times and at least twice,
     \item the integers $\frac{\pi_1}{2}, \frac{\pi_1}{2}+1, \dots, \pi_1$
     all appear at least once, and
     \item there are exactly $\frac{\pi_1}{2}-1$
negative parts, each  $\equiv 1 \pmod{3}$ and $\leqq \frac{3\pi_1}{2}-2$,
with the parts greater than 1 occurring at most once,
  \end{itemize}
equals the number of (ordinary) partitions of $n$ into parts
congruent to $\pm 2, \pm3, \pm 4, \pm 5, \pm 6 \pmod{18}$.
\end{theorem}

\begin{proof}
The general term on the left side of \eqref{m18-1} may be written as
\begin{align*}
\frac{  q^{j^2+j} (-1;q^3)_j } { (-1;q)_j (q;q)_{2j} } &=
\frac{q^{j^2+j}}{(q^2;q^2)_{j-1}}\frac{q^{(3j^2-3j)/2}}{\prod_{k=0}^{j}(1-q^{j+k})}
\prod_{k=1}^{j-1}(1+q^{-3k})\\
&=
\frac{q^{j^2-j}}{(q^2;q^2)_{j-1}}\frac{q^{(3j^2+3j)/2}}{\prod_{k=0}^{j}(1-q^{j+k})}
q^{-j}\prod_{k=1}^{j-1}(1+q^{-3k}).
\end{align*}
\end{proof}

Finally, we provide a combinatorial interpretation of~\eqref{m18-3}
\begin{theorem}\label{t-3}
The number of signed partitions $\sigma = (\pi, \nu)$ of $n+1$,
wherein
  \begin{itemize}
     \item $\pi_1$, the largest positive part, is odd,
     \item the integers $1,2,\dots, \frac{\pi_1-1}{2}$ all appear
     an even number of times and at least twice,
     \item the integers $\frac{\pi_1-1}{2}+1, \frac{\pi_1-1}{2}+2, \dots, \pi_1$
     all appear at least once, and
     \item there are exactly $\frac{\pi_1-1}{2}$
negative parts, each  $\equiv 1 \pmod{3}$ and $\leqq \frac{3\pi_1}{2}-2$,
with the parts greater than 1 occurring at most once,
  \end{itemize}
equals the number of (ordinary) partitions of $n$ into parts
congruent to $1$, $2$, $4$, $5$, $7$ or $8$ modulo $9$, such that for any
nonnegative integer $j$,
$9j+1$ and $9j+2$ do not both appear, and for any nonnegative integer $k$, $9k+7$ and $9k+8$
do not both appear.
\end{theorem}
\begin{proof}
The general term on the left side of~\eqref{m18-3} may be written as
\begin{align*}
\frac{  q^{j(j+1)} (-q^3;q^3)_j}{ (-q;q)_j (q;q)_{2j+1} }  &=
\frac{ q^{j^2 + j}}{(q^2;q^2)_j } \frac{ q^{(3j^2+3j)/2}   }{(q^{j+1};q)_{j+1} } \prod_{k=1}^j (1+q^{-3k}) \\
&= \frac 1q \frac{ q^{j^2 + j}}{(q^2;q^2)_j } \frac{ q^{(3j^2+5j+2)/2}   }{(q^{j+1};q)_{j+1} } q^{-j} \prod_{k=1}^j (1+q^{-3k}),
\end{align*} thus the interpretation of the left side is similar to that of the previous identities.

 The right side of~\eqref{m18-3} provides a challenge because of the double occurrence
of the factors $(q^3;q^{18})_\infty$ and $(q^{15};q^{18})_\infty$ in the numerator.  Accordingly,
we turn to a partition enumeration technique introduced by Andrews and Lewis.
In~\cite[p. 79, Eq. (2.2) with $k=9$]{AL01}, they show that
\begin{equation}
  \frac{ (q^{a+b} ; q^{18} )_\infty}{ (q^a, q^b; q^9)_\infty }
\end{equation}
is the generating function for partitions of $n$ into parts congruent to $a$ or $b$ modulo $9$
such that for any $k$, $9k+a$ and $9k+b$ do not both appear as parts, where $0<a<b<9$.

  With this in mind, we immediately see that
\[ \frac{(q^3,q^6,q^9;q^9)_\infty (q^3,q^{15};q^{18})_\infty}{(q;q)_\infty}
= \frac{1}{(q^4, q^5; q^9)_\infty} \times \frac{ (q^3;q^{18})_\infty }{(q,q^2;q^{9})_\infty}
\times \frac{(q^{15};q^{18})_\infty }{ (q^7, q^8; q^9)_\infty}
\]
generates the partitions stated in our theorem.
\end{proof}

\section*{\normalsize Acknowledgement}
We thank the referee for carefully reading the manuscript and supplying helpful comments, and the conference organizers for the opportunity to present at \emph{Integers 2007}.

\end{document}